\documentclass{article}
\usepackage[utf8]{inputenc}
\usepackage{amsmath}
\usepackage{amsfonts}
\usepackage{amsthm}
\usepackage{amssymb}
\title{Norm inequalities for Calderon-Zygmund  operators in some
generalized Hardy-Morrey spaces}
\author{Martial DAKOURY and Justin FEUTO }

\newtheorem{thm}{Theorem}[section]
\newtheorem{prop}[thm]{Proposition}

\newtheorem{lem}[thm]{Lemma}
\newtheorem{defn}[thm]{Definition}

\newcommand{\comment}[1]{}
\newcommand{\epf}{ $\Box$\medskip}

\def\a{{\mathfrak a}}
\begin{document}

\maketitle

\begin{abstract}
We use a molecular characterization of generalized Hardy-Morrey spaces, to provide a norm controls of Calder\'on-Zygmund operators and their associated commutators in the above mention spaces.
\end{abstract}

\section{Introduction}
 Let $d$ be a positive integer, $\mathbb R^d$ the Euclidean space of dimension $d$ equipped with the Euclidean norm and the Lebesgue measure denoted $dx$. We consider $\varphi$ in the  Schwartz class $\mathcal S(\mathbb R^d)$ having non vanish integral. The classical Hardy space $\mathcal H^p(\mathbb R^d)$ ($0<p<\infty$) consists of tempered distributions $f$ satisfying $\int_{\mathbb R^d}\vert \mathcal Mf(x)\vert^p dx<\infty$,  where 
 \begin{equation}
\mathcal Mf(x)=\sup_{t>0}\vert (f\ast \varphi_t(x)\vert,\label{maxih}
\end{equation}
   with $\varphi_t(x)=t^{-d}\varphi(t^{-1}x)$, $t>0$.
 
 In \cite{AbFe}, Ablé and the second author gave a generalization of these spaces by defining for $0<p,q\leq \infty$ the Hardy-amalgam spaces $\mathcal H^{(p,q ) }( \mathbb R^d)$ as the set of tempered distributions $f$ satisfying  $ \left\|\mathcal Mf\right\|_{p,q}<\infty$, 
where 
$$\left\|g\right\|_{p,q}=\left[\sum_{k\in\mathbb Z^d}\left(\int_{Q_k}\vert g(x)\vert^pdx\right)^{\frac{q}{p}}\right]^{\frac{1}{q}}$$
for any measurable function $g$ and $Q_k=k+\left[0,1\right)^d$, $k\in\mathbb Z^d$. We recall that  Wiener amalgam space $(L^p,\ell^q)(\mathbb R^d)$ is the space of measurable functions $g$ satisfying $\left\|g\right\|_{p,q}<\infty$. From the equality $(L^p,\ell^p)(\mathbb R^d)=L^p$ we have that $\mathcal H^p(\mathbb R^d)$ is a special case of $\mathcal H^{(p,q)}(\mathbb R^d)$. Just like classical Hardy spaces, Hardy-amalgam spaces do not depend on the choice of $\varphi$ in the Schwartz class with nonzero integral.

In \cite{DmFe}, we defined a  generalized Hardy-Morrey space  denoted $\mathcal H^{(p,q,\alpha)}(\mathbb R^d)$ as the sub-space of Hardy-amalgam space $\mathcal H^{(p,q)}(\mathbb R^d)$ consisting of $f\in\mathcal H^{(p,q)}(\mathbb R^d)$ satisfying 
 $$\left\|f\right\|_{\mathcal H^{(p,q,\alpha)}}:=\sup_{\rho>0}\left\|St^\alpha_\rho f\right\|_{\mathcal H^{(p,q)}}<\infty,$$ 
 where $St^\alpha_\rho f$ is the tempered distribution defined by $\langle St^\alpha_\rho f,\varphi\rangle=\langle f,St^{\alpha'}_{\rho^{-1}}\varphi\rangle$, with $\frac{1}{\alpha'}+\frac{1}{\alpha}=1$ and $St^{\alpha}_\rho\varphi(x)=\rho^{\frac{-d}{\alpha}}\varphi(\rho^{-1}x)$ for $\varphi$ in the Schwartz class $\mathcal S(\mathbb R^d)$. 
 
 The space $(L^p,\ell^q)^\alpha(\mathbb R^d)$, consisting of Lebesgue measurable functions $g$ satisfying $\sup_{\rho>0}\left\|St^ \alpha_ \rho f\right\|_{p,q}<\infty$  named Fofana space, was introduced by Fofana Ibrahim in 1988 \cite{Fo1}. For these reason, we will sometimes refer to the space $\mathcal H^{(p,q,\alpha)}(\mathbb R^d)$ as a Hardy-Fofana space. Notice that for $q=\infty$ and $\alpha<p$, $(L^q,\ell^p)^\alpha(\mathbb R^d)$ coincides  with the classical Morrey space $L^{p,d(1-\frac{p}{\alpha}}(\mathbb R^d)$ and $\mathcal H^{(p,\infty,\alpha)}(\mathbb R^d)$ with the Hardy-Morrey space defined by  Jia and Wang in \cite{JW}. This justifies the generalized Hardy-Morrey space terminology.
  
  It is well known that the Hardy space $\mathcal H^p(\mathbb R^d)$ for $p \in\left(0, 1\right]$ is a good substitutes of $L^p(\mathbb R^d)$ when studying the boundedness of operators. We can just cite  the boundedness of 
Riesz operators on $\mathcal H^p(\mathbb R^d)$ but not on $L^p(\mathbb R^d)$ when $p \in\left( 0, 1\right]$. As Wiener amalgam space $(L^p,\ell^q)(\mathbb R^d)$ is concerned (see \cite{H}), its substitutes in this context is the Hardy-amalgam space $\mathcal H^{(p,q)}(\mathbb R^d)$. 

It is also a classical result that the Calder\'on-Zygmund operators are bounded on  $L^{\alpha}$ for
$1 < \alpha < \infty$; see for example   \cite[Theorem 5.10]{JDuo} and \cite[ Theorem 8.2.1]{Graf1}. In \cite{AbFe2}, the authors proved that Calder\'on-Zygmund operators are bounded from Hardy-amalgam spaces $\mathcal H^{(p,q)}(\mathbb R^d)$ into Wiener amalgam spaces $(L^p,\ell^q)(\mathbb R^d)$; and under certain regularity assumptions on its kernel, on Hardy-amalgam space $\mathcal H^{(p,q)}(\mathbb R^d)$.

 From the definition of $\mathcal H^{(p,q,\alpha)}(\mathbb R^d)$ space, we see that $\left\|\cdot\right\|_{\mathcal H^{(p,q)}}\leq \left\|\cdot\right\|_{\mathcal H^{(p,q,\alpha)}}\leq\left\|\cdot\right\|_{\mathcal H^{\alpha}} $, which implies that $\mathcal H^\alpha(\mathbb R^d)\subset\mathcal H^{(p,q,\alpha)}(\mathbb R^d)\subset \mathcal H^{(p,q)}(\mathbb R^d)$, whenever $p\leq \alpha\leq q$.  
 
Our objectives in this article are twofold. We first show that our spaces are stable for some operators known to be bounded in Hardy-amalgam spaces. We can cite among others the Riesz potential operators and Calder\'on-Zygmund operators. Second, we give norm inequalities for commutators associated with the Calder\'on-Zygmund operators. 

The paper is organised as follow :
 Section 2 is devoted to the preliminaries on Hardy-amalgam and generalized Hardy-Morrey spaces. In Section 3 we give a boundedness result for Calder\'on-Zygmund operators in Hardy-Fofana spaces, while in Section 4 we prove the boundedness of their commutators.

Throughout the paper, $\mathcal{S}:= \mathcal S(\mathbb R^{d})$ denotes the Schwartz class of rapidly decreasing smooth functions equipped with its usual topology. The dual space of $\mathcal S$ is the space of tempered distributions denoted by $\mathcal S':= \mathcal S'(\mathbb R^{d})$. If $f\in\mathcal{S'}$ and $\theta\in\mathcal{S}$, we denote the evaluation of $f$ on $\theta$ by $\left\langle f,\theta\right\rangle$. The letter $C$ will be used for non-negative constants independent of the relevant variables, that may change from one occurrence to another. When a constant depends on some important parameters $\alpha,\gamma,\ldots$, we denote it by $C(\alpha,\gamma,\ldots)$. Constants with subscript, such as $C_{s}$, do not change in different occurrences and depend on the parameters mentioned in them. We adopt the following abbreviation $\mathrm{\bf A}\lesssim \mathrm{\bf B}$ for the inequalities $\mathrm{\bf A}\leq C\mathrm{\bf B}$, where $C$ is a non-negative constant independent of the main parameters. If $\mathrm{ A}\lesssim \mathrm{ B}$ and $\mathrm{ B}\lesssim \mathrm{ A}$, then we write $\mathrm{ A}\approx \mathrm{ B}$. The notation $\mathrm{A}:=\mathrm{B}$ means that the definition of $\mathrm{A}$ is $\mathrm{B}$.

For a real number $\lambda>0$ and a cube $Q=(x_Q,\ell_Q)\subset\mathbb R^{d}$ (by a cube we mean a cube whose edges are parallel to the coordinate axes), we write $\lambda Q$ for the cube with same center as $Q$ and side-length $\lambda$ times side-length of $Q$ and $Q^\lambda$ for the cube centered at $\lambda x_Q$ and with side-length $\lambda \ell_Q$ while $\left\lfloor \lambda \right\rfloor$ stands for the greatest integer less than or equal to $\lambda$. Also, for $x\in\mathbb R^{d}$ and a real number $r>0$, $Q(x,r)$ will denote the cube centered at $x$ and side-length $r$. We use the same notations for balls. For a measurable set $E\subset\mathbb R^d$, we denote by $\chi_{E}$ its characteristic function and by $\left|E\right|$ its Lebesgue measure. For a function $f:\mathbb R^d\rightarrow\mathbb C$, $\mathrm{supp }f$ stands for the support of $f$.

\section{Preliminaries on Hardy-amalgam and generalized Hardy-Morrey spaces}
We will always assume that $p\leq \alpha\leq q$, since in \cite{Fo1}, Fofana proved that the space $(L^p,\ell^q)^\alpha(\mathbb R^d) $ is non-trivial if and only if the exponents respect this order.

Just as classical Hardy spaces are related to Lebesgue spaces and Hardy-amalgam spaces to amalgam spaces, we have  the following relationship between our generalized Hardy-Morrey spaces and  Fofana's spaces (see \cite{DmFe}). 

Let $1 \leq p\leq\alpha\leq q < \infty$.
\begin{itemize}
\item If $1<p$ then the spaces $\mathcal H^{(p,q,\alpha)}(\mathbb R^d)$ and $(L^{p},L^{q})^{\alpha}(\mathbb R^d)$ are equal with equivalence norms.
\item The space $\mathcal H^{(1,q,\alpha)}(\mathbb R^d)$ is continuously embedded in $(L^1,\ell^q)^{\alpha}(\mathbb R^d)$.

\item If $p< 1$ then the space $\mathcal H^{(p,q,\alpha)}(\mathbb R^d)$ is a  quasi-Banach space, when it is equipped with  the  quasi-norm
$\left\|\cdot\right\|_{\mathcal H^{(p,q,\alpha)}}$ and for $f,g \in\mathcal H^{(p,q,\alpha)}(\mathbb R^d)$,
$$ \left\|f +g\right\|^{p}_{\mathcal H^{(p,q,\alpha)}}\leq \left\|f\right\|^{p}_{\mathcal H^{(p,q,\alpha)}}+\left\|g\right\|^{p}_{\mathcal H^{(p,q,\alpha)}}.$$ 
\end{itemize}

From the definition of $\mathcal H^{(p,q,\alpha)}(\mathbb R^d)$ spaces, it is clear that if an operator is bounded in $(L^p,\ell^q)(\mathbb R^d)$ respectively in $\mathcal H^{(p,q)}(\mathbb R^d)$ space and commute with the dilation $St^\alpha_\rho$ for all $\rho>0$, then it is also bounded respectively in $(L^p,\ell^q)^\alpha(\mathbb R^d)$ and in $\mathcal H^{(p,q,\alpha)}(\mathbb R^d)$. This is for example the case of Hardy-Littlewood maximal operator defined by 
$$\mathfrak M f(x)=\sup_{r>0}\vert B(x,r)\vert^{-1}\int_{B(x,r)}\vert f(y)\vert dy,$$
 which is bounded in $(L^p,\ell^q)(\mathbb R^d)$ for $p>1$ (see \cite{CHH}) and commute with the dilation $St^\alpha_\rho$; it is  bounded in $(L^p,\ell^q)^\alpha(\mathbb R^d)$. We also have that Riesz transforms $R_j$ defined by 
$$R_jf(x)=\frac{\Gamma(\frac{d+1}{2})}{\pi^{\frac{d+1}{2}}}\mathrm{p.v}\int_{\mathbb R^d}\frac{y_jf(x-y)}{\vert y\vert^{d+1}}dy\ ; j\in\left\lbrace 1,2,\cdots,d\right\rbrace$$
are bounded in $\mathcal H^{(p,q)}(\mathbb R^d)$ for $0<p\leq 1$ as proved in  \cite[Corollary 4.19]{AbFe2}, and they commute with  $St^\alpha_\rho$. Hence, they are bounded in $\mathcal H^{(p,q,\alpha)}(\mathbb R^d)$. As Riesz potential is concerned, we have a result similar to the one obtain in Lebesgue space. 

The Riesz potential operator $I_\gamma$ \ ($0<\gamma<d$), defined by
$$ I_\gamma f(x)=\mathrm{p.v}\int_{\mathbb R^d}\frac{f(y)}{\vert x-y\vert^{d-\gamma}}dy,$$
satisfies the following relation : $I_\gamma \circ St^\alpha_\rho=St^{\alpha^\ast}_\rho\circ I_\gamma$, with $\frac{1}{\alpha^\ast}=\frac{1}{\alpha}-\frac{\gamma}{d}>0$. For $p>1$ the operator $I_\gamma$ is bounded from $(L^p,\ell^q)(\mathbb R^d)$ into $(L^{p^\ast},\ell^{q^\ast})(\mathbb R^d)$ as we can see in \cite{CMP}, so that using the relation between $I_\gamma$ and $St^\alpha_\rho$, we deduce (see \cite[Proposition 4.1]{Fe1}) that $I_\gamma$ is bounded from $(L^p,\ell^q)^\alpha(\mathbb R^d)$ into $(L^{p^\ast},\ell^{q^\ast})^{\alpha^\ast}(\mathbb R^d)$ with a weak control when $p=1$. For $0<p\leq 1$, Abl\'e and the second author proved in \cite[Theorem 4.22]{AbFe2} that if $0<p\leq \min(1,q)<+\infty$ and $\frac{\gamma}{d}<\frac{1}{p}$, then the operator $I_\gamma$ is bounded from $\mathcal H^{(p,q)}(\mathbb R^d)$ into $\mathcal H^{(p^\ast,q^\ast)}(\mathbb R^d)$. We then deduce that  $I_\gamma$ is bounded from $\mathcal H^{(p,q,\alpha)}(\mathbb R^d)$ into $\mathcal H^{(p^\ast,q^\ast,\alpha^\ast)}(\mathbb R^d)$ under the same hypothesis.


In this work, we are particularly interested in the boundedness of Calder\'on-Zygmund operators in generalized Hardy-Morrey spaces. As mention at the beginning of this section,   $\mathcal H^{p,q,\alpha)}(\mathbb R^d)\approx (L^p,\ell^q)^\alpha(\mathbb R^d)$ for $1<p$. The second author proved in  \cite[Theorem 4.5 and 4.6]{Fe1} that Calder\'on-Zygmund operators and their commutators with $BMO$ functions are bounded in $(L^p,\ell^q)^\alpha(\mathbb R^d)$ whenever $1<p$. He also obtain a weak control in the limit case. We prove here that we have some bounded results in $\mathcal H^{(p,q,\alpha)}(\mathbb R^d)$ when $p\leq 1$. In order to achieve our goals, we will need another characterization of our spaces namely the atomic characterization. 


Let $0<p\leq \alpha\leq q<\infty,$  $p\leq 1<r\leq \infty$ and $\alpha\leq r.$  Let  $s$ be an integer greater or equal to $\lfloor d(\frac{1}{p}-1)\rfloor.$  A function $\mathfrak a:\mathbb R^d\rightarrow \mathbb C$ is called  $(p,\alpha,r,s)$-atom if it  satisfies the following conditions:
\begin{enumerate}
\item There exists a cube $Q$ such that 
 $\mathrm{supp}(\a)\subset Q,$
\item
 $\left\|\a\right\|_{r}\leq |Q|^{\frac{1}{r}-\frac{1}{\alpha}};$  
\item
$\int_{\mathbb{R}^{d}}x^{\beta}\a(x)dx=0,$\; for all multi-indexes $\beta$ such that $|\beta|\leq s.$
\end{enumerate}
Notice that if $p=\alpha$ we recover the atoms of $\mathcal H^p(\mathbb R^d)$, which are also the atoms used in \cite{AbFe} in the case of Hardy-amalgam space $\mathcal H^{(p,q)}(\mathbb R^d)$.

We denote $\mathcal{A}(p,\alpha,r,s)$ the set of all $(\a,Q)$ such that $\a$ and $Q$  satisfy conditions $(1)-(3)$. The set $\mathcal A(p,p,r,s)$ is exactly the set $\mathcal A(p,r,s)$ defined in \cite{AbFe}. 
It is easy to see that  $\left( \mathfrak{a},Q\right) \in\mathcal A(p,\alpha,r,s)$ if and only if  $\left( \vert Q\vert^{\frac{1}{\alpha}-\frac{1}{p}}\mathfrak{a},Q\right) \in\mathcal A(p,r,s) $, and $(\mathfrak{a},Q) \in\mathcal A(p,\alpha,r,s)$ implies that  $(St_{\rho}^{\alpha}\mathfrak{a},\rho Q)\in \mathcal A(p,\alpha,r,s)$ which is equivalent to $(St_{\rho}^{\alpha}\mathfrak{a}| Q^\rho|^{\frac{1}{\alpha}-\frac{1}{p}}, Q^\rho)\in\mathcal A(p,r,s)$.

Let $\mathcal H^{(p,q)}_{\mathrm{fin},r,s}(\mathbb R^d)$  be the set of all finite linear combinations of $(p,r,s)$-atoms, equipped with the quasi-norm $\left\|\cdot\right\|_{\mathcal H^{(p,q)}_{\mathrm{fin},r,s}}$ defined by 
$$\Vert f\Vert_{\mathcal H^{(p,q)}_{\mathrm{fin},r,s}}:=\inf\left\{\left\|\sum^{N}_{n=0}\bigg(   \frac{|\lambda_{n}|}{\left\|\chi_{Q_{n}}\right\|_{p}}\bigg)^{\eta}\chi_{Q_{n}}\right\|_{\frac{p}{\eta},\frac{q}{\eta}}^{\frac{1}{\eta}} : f=\sum^N_{n=0}\lambda_{n}\mathfrak{a}_{n} \right\},$$
where the infimum is taken over all finite sequences $\left\lbrace (\mathfrak{a}_n,Q_n)\right\rbrace^N_{n=0}$ in $\mathcal A(p,r,s)$ and all finite sequences of scalars $(\lambda_n)^N_{n=0}$ in $\mathbb C$ such that $f=\sum^N_{n=0}\lambda_{n}\mathfrak{a}_{n}$, and $\eta$ is a fix real number satisfying $0<\eta\leq 1$ if $r=\infty$ and $0<\eta\leq p$ if $\max \left\lbrace q,1\right\rbrace  <r<\infty$. 

Let $0<p<\min(1,q)$ and $\max(1,q)<r\leq \infty$. The quasi-norms $\left\|\cdot\right\|_{\mathcal H^{(p,q)}_{\mathrm{fin},r,s}}$
and $\left\|\cdot\right\|_{\mathcal H^{(p,q)}}$ are equivalent on $\mathcal H^{(p,q)}_{\mathrm{fin},r,s}(\mathbb R^d)$ whenever $r<\infty$ while $\left\|\cdot\right\|_{\mathcal H^{(p,q)}_{\mathrm{fin},\infty,s}}$
and $\left\|\cdot\right\|_{\mathcal H^{(p,q)}}$ are equivalent on $\mathcal H^{(p,q)}_{\mathrm{fin},\infty,s}(\mathbb R^d)\cap \mathcal C(\mathbb R^d)$,  where $\mathcal C(\mathbb R^d)$ stands for the space of  continuous complex values functions on $\mathbb R^d$ \cite[Theorem 4.9]{AbFe}. 

In \cite{DmFe}, the authors accordingly consider the space $\mathcal H^{(p,q,\alpha)}_{\mathrm{fin},r,s}(\mathbb R^d)$ of finite combinations of $(p,\alpha,r,s)$-atoms equipped  with the quasi-norm 
$$\Vert f\Vert_{\mathcal H^{(p,q,\alpha)}_{\mathrm{fin},r,s}}:=\inf\left\{\left\|\sum^{N}_{n=0}\bigg(   \frac{|\lambda_{n}|}{\left\|\chi_{Q_{n}}\right\|_{\alpha}}\bigg)^{\eta}\chi_{Q_{n}}\right\|_{\frac{p}{\eta},\frac{q}{\eta},\frac{\alpha}{\eta}}^{\frac{1}{\eta}} : f=\sum^N_{n=0}\lambda_{n}\a_{n} \right\},$$
and proved that the above remark remains true in this context.

Let us now specify the definition of the Calder\'on-Zygmund operators which will be of interest in this work. 

Let   
$ \Delta:=\{(x,x): x\in\mathbb{R}^{d}\}$ be the diagonal of $\mathbb{R}^{d}\times\mathbb{R}^{d}$. For $\delta>0$ a linear operator $T$ is called $\delta$-Calder\'on-Zygmund operator if it is a bounded operator on $L^2(\mathbb R^d)$ and there exists a function $K:\mathbb{R}^{d}\times\mathbb{R}^{d}\setminus \Delta\rightarrow \mathbb{C}$ named  kernel and $A>0$ satisfying the following conditions. 
\begin{equation}\label{1}
|K(x,y)|\leq A |x-y|^{-d},
\end{equation}
\begin{equation}\label{2}
|K(x,y)-K(x,z)|\leq A \frac{|y-z|^{\delta}}{|x-y|^{d+\delta}},\;\; \text{if}\;\;  |x-y|\geq 2|y-z|,
\end{equation}
\begin{equation}\label{3}
|K(x,y)-K(w,y)|\leq A \frac{|x-w|^{\delta}}{|x-y|^{d+\delta}},\;\; \text{if}\;\;  |x-y|\geq 2|x-w|,
\end{equation}
and for any $f\in L^2(\mathbb R^d)$ with compact support 
\begin{eqnarray*} 
 T(f)(x)=\int_{\mathbb{R}^{d}}K(x,y)f(y)dy, \; x\notin \mathrm{supp}(f). 
\end{eqnarray*}

In \cite{AbFe2} it is proved that if $T$ is a $\delta$-Calder\'on-Zygmund operator then $T$ is extendable to a bounded operator from $\mathcal H^{(p,q)}(\mathbb R^d)$ into $(L^p,\ell^q)(\mathbb R^d)$ whenever $\frac{d}{d+\delta}<p\leq 1$. 
From the remark given at the beginning of this section, we can say that if in addition the kernel $K$ is homogeneous of degree $s<-\frac{d}{\alpha'}$ then $T$ is bounded from $\mathcal H^{(p,q,\alpha)}(\mathbb R^d)$ into $(L^p,\ell^q)^\alpha(\mathbb R^d)$. 

In order to obtain a bounded extends of a $\delta$-Calder\'on-Zygmund operator on $\mathcal H^{(p,q)}(\mathbb R^d)$, Ablé and the second author consider kernels having a certain regularity. More precisely they proved the following result.
\begin{thm}\cite[Theorem 4.7]{AbFe2}\label{able1} Let $T$ be a $\delta$-Calder\'on-Zygmung operator with kernel $K$. We assume that :
\begin{enumerate}
\item there exist an integer $s\geq 0$ and a constant $C_{s}>0$ such that 
\begin{equation}\label{lem1}
|\partial_{y}^{\beta}K(x,y)|\leq \frac{C_{s}}{|x-y|^{d+|\beta|}},
\end{equation}
for all multi-index $\beta$ with $|\beta|\leq d+2s+3$ \quad and\;  all $(x,y)\in \mathbb{R}^{d}\times \mathbb{R}^{d}\backslash\Delta,$
\item if $f\in L^{2}(\mathbb{R}^{d})$ with compact support and $\int_{\mathbb{R}^{d}}x^{\beta}f(x)dx=0,$ for all $|\beta|\leq d+2s+2,$ then
\begin{equation}\label{3395}
\int_{\mathbb{R}^{d}}x^{\beta}T(f)(x)dx=0,
\end{equation}
for all multi-indexes $\beta$ with $|\beta|\leq s.$
\end{enumerate}
Then, $T$ extends to bounded operator on $\mathcal H^{(p,q)}(\mathbb R^d)$ for $\frac{d}{d+s+1}<p\leq 1$.
\end{thm}
Their proof relies essentially on the molecular characterization of $\mathcal H^{(p,q)}(\mathbb R^d).$

\begin{defn}\label{mole1}
Let $1<r\leq \infty$ and $\alpha\leq r$. Assume $s\geq \lfloor d(\frac{1}{p}-1 ) \rfloor$ is an integer. A measurable function $\mathrm{m}$ is $(p,\alpha,r,s)$-molecule centered on a cube $Q$ if the following conditions are fulfilled.
\begin{enumerate}
\item $\left\|\mathrm{m}\chi_{\widetilde{Q}}\right\|_{r}\leq |Q|^{\frac{1}{r}-\frac{1}{\alpha}},$ where $\widetilde{Q}:= 2\sqrt{d}Q,$
\item  $ |\mathrm{m}(x)|\leq |Q|^{-\frac{1}{\alpha}}\left(1+ \frac{|x-x_{Q}|}{\ell(Q)}\right)^{-2d-2s-3}$ for all $ x\notin \widetilde{Q},$ \text{where} $x_{Q}$ and $ \ell(Q)$ denote respectively the center and the side-length of $Q$.
\item  $\int_{\mathbb{R}^{d}} x^{\beta}\mathrm{m}(x)dx=0$ for all multi-indexes $\beta$ with $|\beta|\leq s.$
\end{enumerate}
\end{defn}
The molecules use by Ablé et al in \cite{AbFe} are $(p,r,s):=(p,p,r,s)$-molecules. We denote by $\mathcal{M}\ell (p,\alpha,r,s)$ the set of all $(\mathrm{m},Q)$ in which $\mathrm{m}$ is a $(p,\alpha,r,s)$-molecule centered on $Q$. As in the case of atoms we remark that $\mathcal{M}\ell (p,p,r,s)$ is exactly the set $\mathcal{M}\ell (p,r,s)$ use in \cite{AbFe}. It comes from this definition that $\mathcal{A}(p,\alpha,r,s)\subset\mathcal{M}\ell (p,\alpha,r,s).$
For the proof of  \cite[Theorem 4.7]{AbFe}, the authors use essentially  the following results.

\begin{lem}\cite[Lemma 4.5]{AbFe2}\label{atom_mol}
Let $T$ be a $\delta$-Calder\'on–Zygmund operator with kernel $K$ which satisfy the following conditions 
\begin{enumerate}
\item  There exist an integer $s\geq 0$  and a constant $C_s>0$ such that
\begin{equation}
    \vert\beta_yK(x,y)\vert \leq \frac{C_s}{\vert x-y\vert^{d+\vert\beta\vert}} 
\end{equation}
for all $\beta$ satisfying $\vert\beta\vert\leq d+2s+3$  and all $(x, y) \in\mathbb R^d\times \mathbb R^d\setminus \Delta$.
\item  If $f \in L^2(\mathbb R^d)$ with compact support and 
$\int_{\mathbb R^d}x^\beta f(x)dx=0$, for all multi-indexes $\beta$ with
$\vert\beta\vert \leq d+2s+2$, then
\begin{equation}
    \int_{\mathbb R^d}x^\beta T(f)(x)dx=0
\end{equation}
for all multi-indexes $\beta$ with $\vert\beta\vert\leq s$.
\end{enumerate}
If $d(\frac{1}{p}-1)\leq s$ and $1<r<\infty$, 
then there exists a constant $C_1>0$ such that $(\frac{1}{C_1}T(\a),Q)\in\mathcal M\ell(p,r,s)$ for all $(\a,Q) \in \mathcal A(p,\infty, d+2s +2)$
\end{lem}
This lemma proves that the image of an atom by a $\delta$-Calder\'on-Zygmund operator, is a molecule while the next result will show that appropriate linear combination of molecules of $\mathcal H^{(p,q)}$ is an element of $\mathcal H^{(p,q)}(\mathbb R^d)$. 
\begin{thm}\cite[Theorem 5.3]{AbFe}\label{Able2}
Let $0<p\leq \min(q,1)\leq\max(1,q)<r<\infty$,   $0<\eta\leq p$ and an integer $s\geq\lfloor d(\frac{1}{p}-1)\rfloor$.  For all sequences $\{(\mathrm{m}_{n},Q_{n})\}_{n\geq 0}$ in $\mathcal{M}\ell(p,r,s)$ and all sequences of scalars  $\{\lambda_{n}\}_{n\geq 0}$ such that:
\begin{eqnarray}\label{reconsmolecule2}
\left\|\sum_{n\geq 0}\bigg(   \frac{|\lambda_{n}|}{\left\|\chi_{Q_{n}}\right\|_{p}}\bigg)^{\eta}\chi_{Q_{n}}\right\|_{\frac{p}{\eta},\frac{q}{\eta}}< \infty,
\end{eqnarray}
the series $f:=\sum_{n\geq 0}\lambda_{n}\mathrm{m}_{n}$ converge in  $\mathcal{S'}(\mathbb R^d)$ and $\mathcal{H}^{(p,q)}(\mathbb R^d),$ with
\begin{equation*}
\left\|f\right\|_{\mathcal{H}^{(p,q)}}\lesssim_{\varphi,d,q,p,s}\left\|\sum_{n\geq 0}\bigg(   \frac{|\lambda_{n}|}{\left\|\chi_{Q_{n}}\right\|_{p}}\bigg)^{\eta}\chi_{Q_{n}}\right\|_{\frac{p}{\eta},\frac{q}{\eta}}^{\frac{1}{\eta}}.
\end{equation*}
\end{thm}

\section{Boundedness of Calder\'on-Zygmund opertors in Fofana-Hardy spaces}
To perform in \cite{AbFe2} the proof of the boundedness of Calderon-Zygmund operators in Hardy-amalgam spaces, the authors used estimates that we gather here in a lemma.
\begin{lem}\label{tech}\begin{enumerate}
    \item \cite[Theorem 4.2]{AbFe2} Let $T$ be a $\delta$-Calder\'on-Zygmund operator with kernel $K$, $\frac{d}{d+\delta}<p\leq 1$, $r>\max(2,q)$. If $(\mathfrak a,Q)\in\mathcal A(p,r,s)$ then,
$$\vert T(\mathfrak a)(x)\vert\lesssim \frac{\left[\mathfrak M(\chi_Q)(x)\right]^{\frac{d+\delta}{d}}}{\vert Q\vert^{\frac{1}{p}}}, \  x\notin \tilde{Q}$$
\item \cite[Proposition 4.5]{AbFe} Let $1<\min(u,v)\leq \max(u,v)<s\leq \infty$. For every sequence $(\lambda_n)_n$ of elements of $\mathbb C$ and $\left(b_n\right)_n$ of elements of $L^s(\mathbb R^d)$ such that $\mathrm{supp } \ b_n\subset Q_n$ (\text{a cube }) for all $n$ and $\left\| b_n\right\|_{s}\leq \vert Q_n\vert^{\frac{1}{s}-\frac{1}{u}}$, we have 
$$\left\|\sum_{n\geq 0}\vert \lambda_nb_n\vert\right\|_{u,v}\lesssim \left\|\sum_{n\geq 0}\frac{\vert \lambda_n\vert}{\vert Q_n\vert^{\frac{1}{u}}}\chi_{Q_n}\right\|_{u,v},$$
with the implicit constants not depending on the sequences.
\item \cite[Proposition 11.12]{YlYs} Let $1 <s < \infty$ and $1 < t, u \leq \infty$. Then, for any sequence $(f_n)_n$ of measurable functions, we have
$$\left\|\left[\sum_{n\geq 0}\left(\mathfrak M(f_n)\right)^u\right]^{\frac{1}{u}}\right\|_{s,t}\approx \left\|\left[\sum_{n\geq 0}\vert f_n\vert^u\right]^{\frac{1}{u}}\right\|_{s,t},$$
with the implicit constants not depending on the sequence.
\end{enumerate}
\end{lem}

We are now ready to establish and to prove our result for generalized Hardy-Morrey spaces.
\begin{thm}\label{4}
Let $T$ be a $\delta$-Calder\'on-Zygmund operator.
 The extension of the operator $T$ defined on $\mathcal H^{(p,q)}(\mathbb R^d)$ is bounded from  $\mathcal{H}^{(p,q,\alpha)}(\mathbb R^d)$ into 
 $(L^{p}, \ell^{q})^{\alpha}(\mathbb R^d)$ whenever $\frac{d}{d+\delta}< p \leq 1.$
\end{thm}
\begin{proof}
Let $r > \max \{2, q\}$ and an integer 
$ s\geq \lfloor d(\frac{1}{p}-1)\rfloor$. Fix $f\in \mathcal{H}_{\mathrm{fin},r,s}^{(p,q,\alpha)}(\mathbb R^d)$  and let 
$\{(a_{n},Q_{n})\}_{n=0}^{j}\subset \mathcal{A}(p,\alpha,r,s)$ and $\{\lambda_{n}\}_{n=0}^{j}\subset \mathbb C$ such that  $f =\sum_{n=0}^{j} \lambda_{n}a_{n}.$ We have 
\begin{eqnarray*}\vert T(f)(x)\vert&\leq& \sum_{n=0}^{j} \vert\lambda_{n}\vert \vert T(a_{n})(x)\vert  =\sum_{n=0}^{j} \vert\lambda_{n}\vert\vert Q\vert^{\frac{1}{p}-\frac{1}{\alpha}} \vert T(\vert Q\vert^{\frac{1}{\alpha}-\frac{1}{p}}a_{n})(x)\vert \\
&\lesssim & \sum_{n=0}^{j}|\lambda_{n}| \left(|T(a_{n})(x)\chi_{\widetilde{Q}_{n}}(x)|+\frac{[\mathfrak{M}(\chi_{Q_{n}}(x)) ]^{\frac{d+\delta}{d}}}{\left\|\chi_{Q_{n}}\right\|_{\alpha}} \right)
\end{eqnarray*}
 for all  $x\in\mathbb{R}^{d}$, 
 thanks to Lemma \ref{tech} (1). It follows that 
 
\begin{equation}
\left\|T(f)\right\|_{p,q,\alpha}
\lesssim \left\|\sum_{n=0}^{j}|\lambda_{n}| |T(a_{n})|\chi_{\widetilde{Q}_{n}}\right\|_{p,q,\alpha}+\left\|\sum_{n=0}^{j}|\lambda_{n}|\frac{[\mathfrak{M}(\chi_{Q_{n}}) ]^{\frac{d+\delta}{d}}}{\left\|\chi_{Q_{n}}\right\|_{\alpha}} \right\|_{p,q,\alpha}.\label{estimate}
\end{equation}
Let $0<\eta<p$. Fix $\rho>0$. Since  
$$\mathrm{supp } \vert \tilde{Q_n}^\rho\vert^{\frac{\eta}{\alpha}-\frac{\eta}{p}} St^{\frac{\eta}{\alpha}}_\rho \left(\vert T(a_n)\vert \chi_{\tilde{Q_n}^\rho}\right)^\eta\subset \tilde{Q_n}^\rho$$
 and 
 $$\left\| \vert \tilde{Q_n}^\rho\vert^{\frac{\eta}{\alpha}-\frac{\eta}{p}} St^{\frac{\eta}{\alpha}}_\rho \left(\vert T(a_n)\vert \chi_{\tilde{Q_n}^\rho}\right)^\eta\right\|_{\frac{r}{\eta}}\lesssim \vert \tilde{Q_n}^\rho\vert^{\frac{\eta}{r}-\frac{\eta}{p}},$$
we have
\begin{eqnarray*}\left\|St^\alpha_\rho\left(\sum_{n=0}^{j}|\lambda_{n}| |T(a_{n})|\chi_{\widetilde{Q}_{n}}\right)\right\|_{p,q} &\leq& \left\|\sum_{n=0}^{j}|\lambda_{n}|^\eta St^\alpha_\rho\left(|T(a_{n})|\right)^\eta\chi_{\widetilde{Q}^\rho_{n}}\right\|^{\frac{1}{\eta}}_{\frac{p}{\eta},\frac{q}{\eta}}\\
&\leq&\left\|\sum_{n=0}^{j}(\vert \tilde{Q_n}^\rho\vert^{\frac{\eta}{p}-\frac{\eta}{\alpha}}\vert \lambda_{n}\vert^\eta )\left[\vert \tilde{Q_n}^\rho\vert^{\frac{\eta}{\alpha}-\frac{\eta}{p}}St^\alpha_\rho(|T(a_{n})|)^\eta\chi_{\tilde{Q}^\rho_{n}}\right]\right\|^{\frac{1}{\eta}}_{\frac{p}{\eta},\frac{q}{\eta}}\\
&\lesssim & \left\|\sum_{n=0}^{j}\frac{\vert \tilde{Q_n}^\rho\vert^{\frac{\eta}{p}-\frac{\eta}{\alpha}}\vert \lambda_{n}\vert^\eta}{\left\|\chi_{\tilde{Q_n}^\rho }\right\|_{\frac{p}{\eta}}}\chi_{\widetilde{Q}^\rho_{n}}
\right\|^{\frac{1}{\eta}}_{\frac{p}{\eta},\frac{q}{\eta}}=\left\|\sum_{n=0}^{j}St^{\frac{\alpha}{\eta}}_\rho\left(\frac{\vert \lambda_{n}\vert^\eta}{\left\|\chi_{\tilde{Q_n} }\right\|^\eta_{\alpha}}\chi_{\widetilde{Q}_{n}}\right)
\right\|^{\frac{1}{\eta}}_{\frac{p}{\eta},\frac{q}{\eta}}
\end{eqnarray*}
according to Lemma \ref{tech} (2). 
The above inequality being true for all $\rho>0$, we deduce that  
\begin{equation}
\left\|\sum_{n=0}^{j}|\lambda_{n}| |T(a_{n})|\chi_{\widetilde{Q}_{n}}\right\|_{p,q,\alpha}\lesssim \left\|\sum_{n=0}^{j}\left(\frac{\vert \lambda_{n}\vert}{\left\|\chi_{\tilde{Q_n} }\right\|_{\alpha}}\chi_{\widetilde{Q}_{n}}\right)^\eta\right\|^{\frac{1}{\eta}}_{\frac{p}{\eta},\frac{q}{\eta},\frac{\alpha}{\eta}}. \label{eq1}
\end{equation}
Let us estimate now the second term. Let $\rho>0$ and $\eta$ as above. We have 
\begin{eqnarray*}\left\|St^\alpha_\rho\left(\sum_{n=0}^{j}|\lambda_{n}|\frac{[\mathfrak{M}(\chi_{Q_{n}}) ]^{\frac{d+\delta}{d}}}{\left\|\chi_{Q_{n}}\right\|_{\alpha}} \right)\right\|_{p,q}&=&\left\|\left(\rho^{-\frac{d}{\alpha}}\sum_{n=0}^{j}|\lambda_{n}|\frac{[\mathfrak{M}(\chi_{Q_{n}^\rho}) ]^{\frac{d+\delta}{d}}}{\left\|\chi_{Q_{n}}\right\|_{\alpha}} \right)\right\|_{p,q}\\
&\leq &\left\|\left(\sum_{n=0}^{j}|\lambda_{n}|\vert Q_n^\rho\vert^{\frac{1}{p}-\frac{1}{\alpha}}\frac{[\mathfrak{M}(\chi_{Q_{n}^\rho}) ]^{\frac{d+\delta}{d}}}{\left\|\chi_{Q_{n}^\rho}\right\|_{p}} \right)\right\|_{p,q}\\
&\lesssim &\left\|\left(\sum_{n=0}^{j}\left(\frac{|\lambda_{n}|\vert Q_n^\rho\vert^{\frac{1}{p}-\frac{1}{\alpha}}}{\left\|\chi_{Q_{n}^\rho}\right\|_{p}}\right)^\eta \chi_{Q_{n}^\rho} \right)\right\|^{\frac{1}{\eta}}_{\frac{p}{\eta},\frac{q}{\eta}}\\
&\lesssim&\left\|St^\frac{\alpha}{\eta}_\rho\left(\sum_{n=0}^{j}\left(\frac{|\lambda_{n}|}{\left\|\chi_{Q_{n}}\right\|_{\alpha}}\right)^\eta \chi_{Q_{n}} \right)\right\|^{\frac{1}{\eta}}_{\frac{p}{\eta},\frac{q}{\eta}}
\end{eqnarray*}
where the inequality before last comes from Lemma \ref{tech} (2). We also  deduce that 
\begin{equation}\left\|\sum_{n=0}^{j}|\lambda_{n}|\frac{[\mathfrak{M}(\chi_{Q_{n}}) ]^{v}}{\left\|\chi_{Q_{n}}\right\|_{\alpha}} \right\|_{p,q,\alpha} \lesssim \left\|\left(\sum_{n=0}^{j}\left(\frac{|\lambda_{n}|}{\left\|\chi_{Q_{n}}\right\|_{\alpha}}\right)^\eta \chi_{Q_{n}} \right)\right\|^{\frac{1}{\eta}}_{\frac{p}{\eta},\frac{q}{\eta},\frac{\alpha}{\eta}},\label{eq2}
\end{equation}
by taking successively the supremum of the right-hand side and the left-hand side  over $\rho>0$.
Taking Estimates (\ref{eq1}) and (\ref{eq2}) in (\ref{estimate}), we obtain 
\begin{equation*}
\left\|T(f)\right\|_{p,q,\alpha} \lesssim \left\|\sum_{n=0}^{j}\left(\frac{|\lambda_{n}|}{\left\|\chi_{Q_{n}}\right\|_{\alpha}}\right)^{\eta}\chi_{Q_{n}} \right\|_{\frac{p}{\eta},\frac{q}{\eta},\frac{\alpha}{\eta}}^{\frac{1}{\eta}}.
\end{equation*}
Which leads to 
\begin{eqnarray*}
\left\|T(f)\right\|_{p,q,\alpha}\lesssim \left\|f\right\|_{\mathcal{H}_{\mathrm{fin},r,s}^{(p,q,\alpha)}}\lesssim\left\|f\right\|_{\mathcal{H}^{(p,q,\alpha)}},
\end{eqnarray*}
The result follows from de density of $\mathcal H^{(p,q,\alpha)}_{\mathrm{fin},r,s}(\mathbb R^d)$ in $\mathcal H^{(p,q,\alpha)}(\mathbb R^d)$.
\end{proof}

As Theorem \ref{able1} is concerned, we have the following result in the case of Fofana-Hardy spaces.
\begin{thm}\label{fincalderon}
Let $T$ be a Calder\'on-Zygmund operator satisfying the hypothesis of Theorem \ref{able1}. If  $\frac{d}{d+s+1}<p\leq 1$ and $p\leq \alpha\leq q$  then  $T$ is bounded  on $\mathcal{H}^{(p,q,\alpha)}(\mathbb R^d)$.
\end{thm}
For the proof, we will need the analogue of Theorem \ref{Able2}  
in $\mathcal{H}^{(p,q,\alpha)}(\mathbb R^d)$. 
For this purpose, we use the fact that 
 $\left(\mathfrak{m},Q\right) \in\mathcal M\ell(p,\alpha,r,s)$ if and only if  $\left( \vert Q\vert^{\frac{1}{\alpha}-\frac{1}{p}}\mathfrak{m},Q\right) \in\mathcal M\ell(p,r,s) $, and  that $\left(St^\alpha_\rho \left(\vert Q^\rho\vert^{\frac{1}{\alpha}-\frac{1}{p}}\mathfrak{m}\right),Q^\rho\right) \in\mathcal M\ell(p,r,s)$ whenever $\left(\mathfrak{m},Q\right) \in\mathcal M\ell(p,\alpha,r,s)$, where  $Q^\rho=Q(\rho x_Q,\rho\ell(Q))$. 

\begin{prop}\label{reconstmolecule2}
Let $0<p\leq \min(q,1)\leq\max(1,q)<r<\infty$,   $0<\eta\leq p$ and an integer $s\geq\lfloor d(\frac{1}{p}-1)\rfloor$. For all sequences $\{(\mathrm{m}_{n},Q_{n})\}_{n\geq 0}$ of elements of 
 $\mathcal{M}\ell(p,\alpha,r,s)$ and all sequences of scalars $\{\lambda_{n}\}_{n\geq 0}$ such that 
\begin{equation}\label{reconstmolecule3}
\left\|\sum_{n\geq 0}\bigg(   \frac{|\lambda_{n}|}{\left\|\chi_{Q_{n}}\right\|_{\alpha}}\bigg)^{\eta}\chi_{Q_{n}}\right\|_{\frac{p}{\eta},\frac{q}{\eta},\frac{\alpha}{\eta}}< \infty,
\end{equation}
 the series $f:=\sum_{n\geq 0}\lambda_{n}\mathrm{m}_{n}$  converge in $\mathcal{S'}$ and in $\mathcal{H}^{(p,q,\alpha)},$ with
\begin{eqnarray*}
\left\|f\right\|_{\mathcal{H}^{(p,q,\alpha)}}\lesssim_{d,q,p,s}\left\|\sum_{n\geq 0}\bigg(   \frac{|\lambda_{n}|}{\left\|\chi_{Q_{n}}\right\|_{\alpha}}\bigg)^{\eta}\chi_{Q_{n}}\right\|_{\frac{p}{\eta},\frac{q}{\eta},\frac{\alpha}{\eta}}^{\frac{1}{\eta}}
\end{eqnarray*}
\end{prop}
\proof
From Relation (\ref{reconstmolecule3}), we deduce that the series $\sum_{n\geq 0}\bigg(   \frac{|\lambda_{n}|}{\left\|\chi_{Q_{n}}\right\|_{\alpha}}\bigg)^{\eta}\chi_{Q_{n}}$ converges in   $\left(L^{\frac{p}{\eta}},\ell^{\frac{q}{\eta}}\right)(\mathbb R^d)$. Fix $\rho>0$. We have 
\begin{eqnarray*}St^{\frac{\alpha}{\eta}}_\rho\left(\sum_{n\geq 0}\bigg(   \frac{|\lambda_{n}|}{\left\|\chi_{Q_{n}}\right\|_{\alpha}}\bigg)^{\eta}\chi_{Q_{n}}\right)&=&\sum_{n\geq 0}\bigg(   \frac{|\lambda_{n}|}{\left\|\chi_{Q_{n}}\right\|_{\alpha}}\bigg)^{\eta}St^{\frac{\alpha}{\eta}}_\rho\chi_{Q_{n}}\\
&=&\sum_{n\geq 0}\bigg(  \rho^{-\frac{d\eta}{\alpha}} \frac{|\lambda_{n}|}{\left\|\chi_{Q_{n}}\right\|_{\alpha}}\bigg)^{\eta}\chi_{Q^\rho_{n}}.
\end{eqnarray*}
It comes that for $\rho>0$, we have
\begin{eqnarray*}
\left\|St^{\frac{\alpha}{\eta}}_\rho\left[\sum_{n\geq 0}\bigg(   \frac{|\lambda_{n}|}{\left\|\chi_{Q_{n}}\right\|_{\alpha}}\bigg)^{\eta}\chi_{Q_{n}}\right]\right\|_{\frac{p}{\eta},\frac{q}{\eta}}&=&\left\|\sum_{n\geq 0}\bigg(  \rho^{-\frac{d\eta}{\alpha}} \frac{|\lambda_{n}|}{\left\|\chi_{Q_{n}}\right\|_{\alpha}}\bigg)^{\eta}\chi_{Q^\rho_{n}}\right\|_{\frac{p}{\eta},\frac{q}{\eta}}\\
&=&\left\|\sum_{n\geq 0}\bigg(  \frac{\vert Q^\rho_n\vert^{\frac{1}{p}-\frac{1}{\alpha}}|\lambda_{n}|}{\left\|\chi_{Q^\rho_{n}}\right\|_{p}}\bigg)^{\eta}\chi_{Q^\rho_{n}}\right\|_{\frac{p}{\eta},\frac{q}{\eta}}<\infty.
\end{eqnarray*}
Since $\left(St^\alpha_\rho(\vert Q^\rho_n\vert^{\frac{1}{\alpha}-\frac{1}{p}}\mathrm{m}_n),Q^\rho_n\right)\in\mathcal M\ell(p,r,s)$, the series 
$$\sum_{n\geq 0}\lambda_n\vert Q^\rho_n\vert^{\frac{1}{p}-\frac{1}{\alpha}}St^\alpha_\rho(\vert Q^\rho_n\vert^{\frac{1}{\alpha}-\frac{1}{p}}\mathrm{m}_n)=St^\alpha_\rho(\sum_{n\geq 0}\lambda_{n}\mathrm{m}_{n})$$
converges in $\mathcal S'$ and in $\mathcal H^{(p,q)}$, thanks to Proposition \ref{reconstmolecule2}. Furthermore, 
\begin{eqnarray*}\left\|St^\alpha_\rho(\sum_{n\geq 0}\lambda_{n}\mathrm{m}_{n})\right\|_{p,q}&\lesssim_{d,q,p,s}&\left\|St^{\frac{\alpha}{\eta}}_\rho\left[\sum_{n\geq 0}\bigg(   \frac{|\lambda_{n}|}{\left\|\chi_{Q_{n}}\right\|_{\alpha}}\bigg)^{\eta}\chi_{Q_{n}}\right]\right\|_{\frac{p}{\eta},\frac{q}{\eta}}\\
&\lesssim_{d,q,p,s}&\left\|\left[\sum_{n\geq 0}\bigg(   \frac{|\lambda_{n}|}{\left\|\chi_{Q_{n}}\right\|_{\alpha}}\bigg)^{\eta}\chi_{Q_{n}}\right]\right\|_{\frac{p}{\eta},\frac{q}{\eta},\frac{\alpha}{\eta}}
\end{eqnarray*}
The result follows.
\epf

\begin{proof}[Proof of Theorem \ref{fincalderon}]
Let $f\in\mathcal H^{(p,q,\alpha)}(\mathbb R^d)$. Since  $\mathcal H^{(p,q,\alpha)}(\mathbb R^d)$ is a subset of $\mathcal H^{(p,q)}(\mathbb R^d)$, we have that $Tf\in \mathcal H^{(p,q)}(\mathbb R^d)$ and $\left\| Tf\right\|_{\mathcal H^{(p,q)}}\lesssim \left\| f\right\|_{\mathcal H^{(p,q)}}$, thanks to Theorem \ref{able1}.

 Let $s\geq \lfloor d(\frac{1}{p}-1) \rfloor$ be an integer;    $s_{0}=d+2s+2$ \; and  \; $f\in \mathcal{H}_{\mathrm{fin},\infty,s_{0}}^{(p,q,\alpha)}\cap \mathcal{C}(\mathbb{R}^{d}).$ We consider a finite sequence $\{(\mathfrak a_{n},Q_{n})\}_{n=0}^{j}$ of elements of $\mathcal{A}(p,\alpha,\infty,s_0)$ and a finite sequence of scalars $\{\lambda_{n}\}_{n=0}^{j}$ such that $f=\sum_{n=0}^{j}\lambda_{n}\mathfrak a_{n}.$ Since for all $n$ we have $(\vert Q_n\vert^{\frac{1}{\alpha}-\frac{1}{p}}\mathfrak a_n,Q_n)\in\mathcal A(p,\infty,s_0)$, we have that $(\frac{1}{C_1}T(\vert Q_n\vert^{\frac{1}{\alpha}-\frac{1}{p}}\mathfrak a_n),Q_n)\in\mathcal M\ell(p,r,s)$ whenever  $\max\{1,p\}<r<+\infty;$ thanks to Lemma \ref{atom_mol}. That is  $(\frac{1}{C_1}T(\mathfrak a_n),Q_n)\in\mathcal M\ell(p,\alpha,r,s)$.
 Let's fix $0<\eta<p.$ 
We have  
\begin{equation*}
T(f)=\sum_{n=0}^{j}\lambda_{n}T(a_{n})=\sum_{n=0}^{j}(C_{1}\lambda_{n})\left(\frac{1}{C_{1}}T(a_{n})\right) \in \mathcal{H}^{(p,q,\alpha)}
\end{equation*}
 and
\begin{eqnarray*}
\left\|T(f)\right\|_{\mathcal{H}^{(p,q,\alpha)}}\lesssim \left\|\sum_{n=0}^{j}\left(\dfrac{|c_{1}\lambda_{n}|}{\left\|\chi_{Q_{n}}\right\|_{\alpha}}\right)^{\eta}\chi_{Q_{n}} \right\|_{\frac{q}{\eta},\frac{p}{\eta},\frac{\alpha}{\eta}}^{\frac{1}{\eta}}\lesssim\left\|\sum_{n=0}^{j}\left(\dfrac{|\lambda_{n}|}{\left\|\chi_{Q_{n}}\right\|_{\alpha}}\right)^{\eta}\chi_{Q_{n}} \right\|_{\frac{q}{\eta},\frac{p}{\eta},\frac{\alpha}{\eta}}^{\frac{1}{\eta}}
\end{eqnarray*}
according to  Theorem \ref{reconstmolecule2}. This allow us to say that 
\begin{eqnarray*}
\left\|T(f)\right\|_{\mathcal{H}^{(p,q,\alpha)}}\lesssim \left\|f\right\|_{\mathcal{H}_{\mathrm{fin},\infty,s_{0}}^{(p,q,\alpha)}}\lesssim \left\|f\right\|_{\mathcal{H}^{(p,q,\alpha)}}.
\end{eqnarray*}
 Therefore, $T$ is a bounded operator from $\mathcal{H}_{\mathrm{fin},\infty,s_{0}}^{(p,q,\alpha)}\cap \mathcal{C}(\mathbb{R}^{d})$ to $\mathcal{H}^{(p,q,\alpha)}$ and the density of  $\mathcal{H}_{\mathrm{fin},\infty,s_{0}
}^{(p,q,\alpha)}\cap \mathcal{C}(\mathbb{R}^{d})$ in $\mathcal{H}^{(q,p,\alpha)}$ yields the result.

\end{proof}

This result generalized the analogue result in classical Hardy space, and proved that the space $\mathcal H^{(p,q,\alpha)}$ is a stable subspace under $T$, of the
 Hardy-amalgam space $\mathcal H^{(p,q)}$ defined in \cite{AbFe}.  

\section{Norm inequalities for commutators}
Let $T$ be a Calder\'on-Zygmund  operator  and  $\mathfrak{b}$  a locally integrable  function on $\mathbb{R}^{d}.$ The commutator of $T$ and 
  $\mathfrak{b}$ is the operator $[\mathfrak{b},T] $ formally defined by
  \begin{equation}\label{BMO1}
  [\mathfrak{b},T](f)(x)= \mathfrak{b}(x)T(f)(x)-T(\mathfrak{b}f)(x)
  \end{equation}
  for all $x\in\mathbb{R}^{d}.$
  
 Coifman et al \cite{CRW} proved that the commutator $\left[T,\mathfrak b\right]$ is bounded on Lebesgue space for $p\in\left(1,\infty\right)$ when $\mathfrak b\in BMO(\mathbb R^d)$.
  We recall that a locally integrable function $\mathfrak b$ belongs to $BMO(\mathbb R^d)$ if 
  $$\sup_{Q:\text{ cube}}\vert Q\vert^{-1}\int_Q\vert \mathfrak b(x)-\mathfrak b_Q\vert dx<\infty$$
  where $\mathfrak b_Q=\vert Q\vert ^{-1}\int_Q\mathfrak b(x)dx$.

But when $\mathfrak b\in BMO(\mathbb R^d)$, it is well known that the commutator $\left[T,\mathfrak b\right]$ is not bounded from $\mathcal H^1(\mathbb R^d)$ to $L^1(\mathbb R^d)$ if $\mathfrak b$ is not a constant function (see \cite{HST}). Many authors have constructed subspaces of $BMO(\mathbb R^d)$ which ensure the boundedness of $\left[T,\mathfrak b\right]$ from the weighted Hardy spaces to the weighted Lebesgue space (see for example \cite{HK,LKY}). In this article we give a subspace of $BMO(\mathbb R^d)$ which ensures boundedness of commutators from generalized Hardy-Morrey spaces into Fofana spaces.

%
%
%
 We introduced in \cite{DmFe} a subspace of $BMO(\mathbb R^d)$ name $BMO^{d}:=BMO^{d}(\mathbb{R}^{d})$ which consists of  $\mathfrak b\in BMO$ for which there exists $0<\mu_{b}<d$ such that  
   \begin{eqnarray}\label{Jesus}
   \left|\mathfrak{b}(x)-\mathfrak{b}_{Q}\right|\leq C\bigg(\ell_{Q}^{-1}|x-x_{Q}| \bigg)^{\mu_{b}}
   \end{eqnarray}
 for all cubes $Q:=Q(x_{Q},\ell_{Q})$ and all $x\notin Q$. 
 
 Notice that $2\ell^{-1}_Q\vert x-x_Q\vert\geq 1$ for $x\notin Q$. It follows that $ L^\infty(\mathbb R^d)\subset BMO^d$
 
  We also proved in \cite{DmFe} that commutators of intrinsic area function with elements of $BMO^d$ are bounded in $\mathcal H^{(p,q,\alpha)}(\mathbb R^d)$.
 
 For Calder\'on-Zygmund operators, we have the following result.
 \begin{thm}\label{Agosty}
Let $\delta>0$, $T$ a $\delta$-Calder\'on-Zygmund operator and $\mathfrak{b}\in BMO^{d}$. We assume that the kernel of $T$ satisfies the following additional condition : $K(x,z)\leq \frac{1}{2}K(x,y)$ whenever $ 2|x - z| \geq |x - y|$. If $ \frac{d}{d+\delta-\mu_{b}}<p\leq 1$, then $[\mathfrak{b},T]$ extends to a bounded operator from $\mathcal{H}^{(p,q,\alpha) }$ into $(L^{p},\ell^{q})^{\alpha}$.
   
\end{thm}
\begin{proof}
Let $r>\max\{2,q\}$ and an integer $s\geq
\lfloor d(\frac{1}{p}-1)\rfloor$. We fix  $f=\sum_{n=0}^{j}\lambda_{n }a_{n}\in \mathcal{H}_{\mathrm{fin},r,s}^{(p,q,\alpha)}$ with $\mathrm{supp } \ a_n\subset Q_n$ so that  $\{(a_{n},Q_{n})\}_{n=0}^{j}\subset\mathcal{A}(p,\alpha,r, s)$, and  $\{\lambda_{n}\}_{n=0}^{j}$ a finite family of scalars. Put $\widetilde{Q}_{n} :=2\sqrt{d}Q_{n}$ for $n\in \{0,\cdots,j\}$ and let  $x_{n}$ and $\ell_{n}$ be respectively the center and the side length of $Q_{n}.$ We have:
\begin{eqnarray}\label{opayu}
|[\mathfrak{b},T](f)(x)|\leq \sum_{n=0}^{j}|\lambda_{n}||[\mathfrak{b},T](a_{ n})\chi_{\widetilde{Q}_{n}}(x)|+ \sum_{n=0}^{j}|\lambda_{n}||[\mathfrak{b},T]( a_{n})\chi_{\mathbb{R}^{d}\backslash\widetilde{Q}_{n}}(x)|
\end{eqnarray}
for all $x\in\mathbb{R}^{d}.$
Fix $0<\eta<p$. Since $\mathfrak{b}\in BMO(\mathbb R^d)$ and the operator $[\mathfrak{b},T]$ is bounded on $L^{\nu}$ for all $\nu>1,$ we have 

$$\left\|\left(([\mathfrak{b},T](a_{n})\chi_{\widetilde{Q}_{n}}\right)^{\eta}\right\|_{\frac{r}{\eta}}\leq C_{b,r,\eta}\left\|a_{n}\right\|_{r}^{\eta}\leq C_{b,r,\eta,\alpha, d}|\widetilde{Q}_{n}|^{\frac{1}{\frac{r}{\eta}}-\frac{1}{\frac{\alpha}{\eta}}}.$$
It follows that 
\begin{eqnarray*}
\left\|\sum_{n=0}^{j}|\lambda_{n}||[\mathfrak{b},T](a_{n})\chi_{\widetilde{Q}_{n} }|\right\|_{p,q,\alpha}&\leq&\left\|\sum_{n=0}^{j}|\lambda_{n}|^{\eta}\bigg(|[ \mathfrak{b},T](a_{n})\chi_{\widetilde{Q}_{n}}|\bigg)^{\eta}\right\|_{\frac{p}{\eta },\frac{q}{\eta},\frac{\alpha}{\eta}}^{\frac{1}{\eta}}\\&\leq& C_{b,r,\eta,\alpha,d}\left\|\sum_{n=0}^{j}\bigg( \frac{|\lambda_{n}|}{\left\| \chi_{Q_{n}}\right\| _{\alpha}}\bigg)^{\eta}\chi_{Q_{n}}\right\|_{\frac{p}{\eta},\frac{q}{\eta},\frac {\alpha}{\eta}}^{\frac{1}{\eta}}
\end{eqnarray*}
thanks to assertion 2) of Lemma \ref{tech}. For the second term, we fix  $x\notin \widetilde{Q}_{n}.$ 
We have
\begin{eqnarray*}
|[\mathfrak{b},T](a_{n})(x)|&=&|T((\mathfrak{b}(x)-\mathfrak{b})a_{n})(x) |\\&\leq& 2 \int_{Q_{n}} \left|K(x,y)-K(x,x_{n})\right||\mathfrak{b}(x)-\mathfrak{ b}(y)||a_{n}(y)|dy \\&\lesssim& \frac{\ell_{n}^{\delta}}{|x-x_{ n}|^{d+\delta}}\int_{Q_{n}}|\mathfrak{b}(x)-\mathfrak{b}(y)||a_{n}(y)|dy
\end{eqnarray*}
the last inequality is due to the relation (\ref{2}) since $|x-x_{n}|>2|y-x_{n}|$ and to the fact that $|y-x_{n}|\leq \frac{\sqrt{d}}{2}\ell_{n},$ for all $y\in Q_{n}$ while the second inequality comes from the additional hypothesis on the kernel.

  
But then, 
\begin{eqnarray*}
\int_{Q_{n}}|\mathfrak{b}(x)-\mathfrak{b}(y)||a_{n}(y)|dy&\leq& \int_{Q_{n}}(|\mathfrak{b}(x)-\mathfrak{b}_{Q_{n}}|+ |\mathfrak{b}(y)-\mathfrak{b}_{Q_{n}}| )|a_{n }(y)|dy \\&\leq& \int_{Q_{n}}|\mathfrak{b}(x)-\mathfrak{b}_{Q_{n}}||a_{n}(y) |dy+ \int_{Q_{n}}|\mathfrak{b}(y)-\mathfrak{b}_{Q_{n}}| |a_{n}(y)|dy \\&\leq& C \frac{\ell_{n}^{-\mu_{b}}}{|x-x_{n}|^{-\mu_{b }}}|Q_{n}|^{-\frac{1}{\alpha}+1} +\left\|\mathfrak{b}\right\|_{BMO}|Q_{n}|^{ -\frac{1}{\alpha}+1}
\end{eqnarray*}
by condition (\ref{Jesus}) and Hölder's inequality. It comes that

\begin{eqnarray*}
\left|\left[\mathfrak{b},T\right](a_{n})(x)\right|\lesssim \bigg( \frac{\bigg(\mathfrak{M }(\chi_{Q_{n}})(x)\bigg)^{\nu}}{\left\|\chi_{Q_{n}}\right\|_{\alpha}} +\left\|\mathfrak{b}\right\|_{BMO} \frac{\bigg(\mathfrak{M}(\chi_{Q_{n}})(x)\bigg)^{v}}{\left\|\chi_{Q_{n}}\right\|_{\alpha}}\bigg)
\end{eqnarray*}
with $\nu=\frac{d+\delta-\mu_{b}}{d}$ \;and\; $v=\frac{d+\delta}{d}.$
Finally, we have 
\begin{eqnarray*}
\left\|\sum_{n=0}^{j}|\lambda_{n}||[\mathfrak{b},T](a_{n})\chi_{\mathbb{R}^{d} \backslash\widetilde{Q}_{n}}|\right\|_{p,q,\alpha}\leq C(b,T)( E+F)
\end{eqnarray*}
with
\begin{eqnarray*}
E=\left\|\sum_{n=0}^{j}|\lambda_{n}|\frac{\bigg(\mathfrak{M}(\chi_{Q_{n}})(x)\bigg )^{\nu}}{\left\|\chi_{Q_{n}}\right\|_{\alpha}} \right\|_{p,q,\alpha} \text{ and } \; \; F=\left\|\sum_{n=0}^{j}|\lambda_{n}|\frac{\bigg(\mathfrak{M}(\chi_{Q_{n}})(x)\bigg )^{v}}{\left\|\chi_{Q_{n}}\right\|_{\alpha}} \right\|_{p,q,\alpha}.
\end{eqnarray*}
  Let $0<\eta<p.$ It comes from Estimate (\ref{eq2}) that 
\begin{eqnarray*}
\left\|\sum_{n=0}^{j}|\lambda_{n}||[\mathfrak{b},T](a_{n})\chi_{\mathbb{R}^{d} \backslash\widetilde{Q}_{n}}|\right\|_{p,q,\alpha}\leq C(b,T) \left\|\sum_ {n=0}^{j}\bigg(\frac{|\lambda_{n}|}{\left\|\chi_{Q_{n}}\right\|_{\alpha}}\bigg)^ {\eta}\chi_{Q_{n}} \right\|_{\frac{p}{\eta},\frac{q}{\eta},\frac{\alpha}{\eta}}^ {\frac{1}{\eta}}.
\end{eqnarray*}
Therefore, we have
\begin{eqnarray*}
\left\|[\mathfrak{b},T](f)\right\|_{p,q,\alpha}\leq C(b,T) \left\|\sum_{n=0}^{j}\bigg(\frac{|\lambda_{n}|}{\left\|\chi_{Q_{n}}\right\|_{\alpha}}\bigg)^{\eta}\chi_{Q_{n}} \right\|_{\frac{p}{\eta},\frac{q}{\eta},\frac{\alpha}{\eta}}^{\frac{1}{\eta}}
\end{eqnarray*}
according to the relation (\ref{opayu}). Subsequently,
\begin{equation*}
\left\|[\mathfrak{b},T](f)\right\|_{p,q,\alpha} \leq C(b,T) \left\|f\right\|_{\mathcal{H}_{\mathrm{fin},r}^{(q,p,\alpha)}}\leq C(b,T) \left\|f\right\|_{\mathcal{H}^{(p,q,\alpha)}}
\end{equation*}
The result follows from the density of $\mathcal{H}_{\mathrm{fin},r,s}^{(p,q,\alpha)}(\mathbb R^d)$ in $\mathcal{H}^{(p,q,\alpha)}(\mathbb R^d)$.
\end{proof}

With some regularity conditions on the kernel $K$ of $T$, we can remove the additional condition imposed in the above theorem. More precisely, we have the following result.
 \begin{thm}
 Let $\delta>0$ and $T$ a $\delta$-Calder\'on-Zygmund operator associated with a kernel $K$. We assume that there is a positive integer $m\geq d$ and two constants $B_{m}>0$ and $C_{m}>0$ such that:
\begin{equation}\label{Agosty1}
|\partial_{y}^{\beta}K(x,y)|\leq \frac{C_{m}}{|x-y|^{d+|\beta|}}
\end{equation}
and
\begin{equation}\label{agostie1}
\sum_{|\beta|\leq m-1}\frac{1}{\beta!}\frac{|y-z|^{|\beta|}}{|x-z|^{d+|\beta|}} \leq B_{m}|K(x,y)-K(x,z)|
\end{equation}
for all multi-index $\beta$ with $|\beta|\leq m$ and  $(x,y),(x,z)\in \mathbb{R}^{d}\times\mathbb{R }^{d}\backslash\bigtriangleup.$
If $\mathfrak{b}\in BMO^{d}$ and $\max\{\frac{d}{d+\delta-\mu_{b}},\frac{ d }{d+m-\mu_{b}}\}<p\leq 1$, then $[\mathfrak{b},T]$ extends to a bounded operator from $ \mathcal{H}^{(p,q,\alpha) }$ into $(L^{p},\ell^{q})^{\alpha}$. 
\end{thm}
\begin{proof}
Let $b\in BMO^{d}$ and $\max\{\frac{d}{d+\delta-\mu_{b}},\frac{d}{d+m-\mu_{b}}\}<p\leq 1$. Let $r>\max\{2,q\}$  and $s\geq m-\lfloor \mu_{b}\rfloor.$ We fix  $f=\sum_{n=0} ^{j}\lambda_{n}a_{n}\in\mathcal{H}_{\mathrm{fin},r,s}^{(p, q,\alpha)}$ with $\mathrm{ supp }a_n\subset Q_n$ such that $(a_n,Q_n)\in\mathcal A(p,\alpha,r,s)$.
We have
\begin{eqnarray*}
\left\|[\mathfrak{b},T](f)\right\|_{p,q,\alpha}\lesssim \left\|\sum_{n=0}^{j}|\lambda_{n}||[\mathfrak{b},T](a_{ n})\chi_{\widetilde{Q}_{n}}\right\|_{p,q,\alpha}+ \left\|\sum_{n=0}^{j}|\lambda_{n}||[\mathfrak{b},T]( a_{n})\chi_{\mathbb{R}^{d}\setminus\widetilde{Q}_{n}}\right\|_{p,q,\alpha}
\end{eqnarray*}
for all $x\in \mathbb{R}^{d}$, with $\widetilde{Q}_{n} :=2\sqrt{d}Q_{n}$, $0\leq n\leq j$. 

It suffices to estimate the second term, since the first can be estimated exactly as we did in Theorem \ref{Agosty}.

We consider, for a fix $y\in\mathbb R^d$, the Taylor expansion of order $m$ with integral remainder  of the map $x\mapsto K(x,y)$ in the neighborhood of $x_n$. We denote by $R(x_n,y)$ the remainder. 

%
Let $0<\eta<p.$ We have 
\begin{eqnarray*}
|[\mathfrak{b},T](a_{n})\chi_{\mathbb{R}^{d}\backslash\widetilde{Q}_{n}}(x)|& \leq&\int_{ Q_{n}}|K(x,y)| |(\mathfrak{b}(x)-\mathfrak{b}(y))||a_{n}|(y)dy\\&\leq& \int_{Q_{n}}\bigg[B_{m }C_{m}|K(x,y)-K(x,x_{n})| + |R(y,x_{n})| \bigg] |\mathfrak{b}(x)-\mathfrak{b}(y)||a_{n}(y)|dy.
\end{eqnarray*}
%
%
%
Since $y\in Q_n$ and $x\notin \tilde{Q_n}$, we have 
$|x-x_{n}|>2|y-x_{n}|$ so that 
$$|R(y,x_{n})|\leq C(m)\dfrac{|y-x_{n}|^{m}}{|x-x_{n}|^ {d+m}}, \  |y-x_{n}|\lesssim\ell_{n}$$
and
\begin{eqnarray*}
|K(x,x_{n})-K(x,y)|\leq A \frac{|y-x_{n}|^{\delta}}{|x-x_{n}|^{d+ \delta}}.
\end{eqnarray*}
Hence
\begin{eqnarray*}
|[\mathfrak{b},T](a_{n})\chi_{\mathbb{R}^{d}\setminus\widetilde{Q}_{n}}(x)|\leq C\bigg( \frac{\ell_{n}^{m}}{|x-x_{n}|^{d+m}}+\frac{\ell_{n}^{\delta}}{|x-x_{ n}|^{d+\delta}}\bigg)\int_{Q_{n}}|\mathfrak{b}(x)-\mathfrak{b}(y)||a_{n}(y)|dy.
\end{eqnarray*}
Moreover
\begin{eqnarray*}
\int_{Q_{n}}|\mathfrak{b}(x)-\mathfrak{b}(y) ||a_{n}(y)|dy\leq \bigg( \frac{\ell_{n} ^{-\mu_{b}}}{|x-x_{n}|^{-\mu_{b}}} +\left\|\mathfrak{b}\right\|_{BMO}\bigg) |Q_{n}|^{-\frac{1}{\alpha}+1}
\end{eqnarray*}
  and
\begin{eqnarray*}
  \frac{\ell_{n}}{|x-x_{n}|}\leq C_{d,s}\bigg(\mathfrak{M}(\chi_{Q_{n}})(x)\bigg )^{\frac{1}{d}}\;\; x\notin Q_{n},
\end{eqnarray*}
so that 
\begin{eqnarray}\label{arafat1}
|[\mathfrak{b},T](a_{n})\chi_{\mathbb{R}^{d}\backslash\widetilde{Q}_{n}}(x)|\leq C \sum_{ i=1}^{4} \frac{\bigg(\mathfrak{M}(\chi_{Q_{n}})(x)\bigg)^{v_{i}}}{\left\|\chi_ {Q_{n}}\right\|_{\alpha}}
\end{eqnarray}
with \;
$v_{1}={\frac{d+m-\mu_{0}}{d}},$ \; $v_{2}={\frac{d+m}{d}},$ \; $v_{3}={\frac{d+\delta-\mu_{0}}{d}},$ \; $v_{4}={\frac{d+\delta}{d}}$\; and  $C:=C(m,T,b).$

As for the previous result, we have 
\begin{eqnarray*}
\left\|\sum_{n=0}^{j}|\lambda_{n}||[\mathfrak{b},T]( a_{n})\chi_{\mathbb{R}^{d}\setminus\widetilde{Q}_{n}}\right\|_{p,q,\alpha}\leq C_{b,\delta,r,\eta,m,d,\mu_{0}}\left\|\sum_{n=0}^{j}\bigg( \dfrac{|\lambda_{ n}|}{\left\| \chi_{Q_{n}}\right\|_{\alpha}}\bigg)^{\eta}\chi_{Q_{n}}\right\|_{\frac{p}{\eta},\frac{q}{\eta},\frac{\alpha}{\eta}}^{\frac{1}{\eta}}
\end{eqnarray*}

The rest of the proof is then similar to that of the proof of Theorem \ref{Agosty}.
\end{proof}
The next result show that for some nice conditions on the kernel, our spaces are stable under the action of $\delta$-Calder\'on-Zygmund operator.
  \begin{thm}\label{goupiste4}
Let  $0<p\leq 1$, $\mathfrak{b}\in BMO^{d}$,  $\delta> 2\mu_{b}+ 2 d(\frac{1}{p}-1)+d+5$ and $T$ a $\delta$-Calder\'on-Zygmund operator with kernel $K$. Assume that :
\begin{itemize}
    \item $K$ verifies (\ref{Agosty1}) and (\ref{agostie1}),
    \item if $f\in L^{\nu}(\mathbb{R}^{d})$ \ $(\nu\geq \max\{2,r\})$ have compact support and  satisfies $\int_{\mathbb{R}^{d}}x^{\beta}f(x)dx=0$ for all $|\beta|\leq d+2s+2$ then
$$\int_{\mathbb{R}^{d}}x^{\beta}[\mathfrak{b},T](f)(x)dx=0$$
for all multi-indexes $\beta$ with $|\beta|\leq s.$
\end{itemize} 
Then $[\mathfrak{b},T]$ is a  bounded operator on $\mathcal{H}^{(p,q,\alpha)}.$
\end{thm}
For the proof, we need the following lemma.
\begin{lem}\label{goupiste3}
Let $0<p\leq 1$, $\mathfrak{b}\in BMO^{d}$,   $\delta> 2\mu_{b}+ 2 d(\frac{1}{p}-1)+d+5$ and $T$ a $\delta$-Calder\'on-Zygmund operator with kernel  $K$. We assume the following properties  valid.
\begin{enumerate}
\item  There exist  an integer $s\in \left[\lfloor\mu_{b}\rfloor +\lfloor d(\frac{1}{p}-1)\rfloor +1,  \lfloor \frac{\delta -d-3}{2}\rfloor\right]$ and two  constants $C_{s}>0$,  $B_{s}>0$ such that 
\begin{equation}\label{A-A}
|\partial_{y}^{\beta}K(x,y)|\leq \frac{C_{s}}{|x-y|^{d+|\beta|}}
\end{equation}
and
\begin{equation}\label{jojo2bis}
\sum_{|\beta|\leq d+2s+2}\frac{1}{\beta!}\frac{|y-z|^{|\beta|}}{|x-z|^{d+|\beta|}} \leq B_{s}|K(x,y)-K(x,z)|
\end{equation}

for all multi-indexes $\beta$ with $|\beta|\leq d+2s+3$ and  for all $(x,y)$ and $(x,z)$ in  $\mathbb{R}^{d}\times \mathbb{R}^{d}\backslash\Delta.$
\item If $f\in L^{\nu}(\mathbb{R}^{d})$ \ $(\nu\geq \max\{2,r\})$ have compact support and $\int_{\mathbb{R}^{d}}x^{\beta}f(x)dx=0$ for all $|\beta|\leq d+2s+2$, then
\begin{equation}\label{3395bis}
\int_{\mathbb{R}^{d}}x^{\beta}[\mathfrak{b},T](f)(x)dx=0
\end{equation}
for all multi-indexes $\beta$ with $|\beta|\leq s.$
\end{enumerate}
 Then for all $(a,Q)\in \mathcal{A}(q,\alpha,\infty,d+2s+2)$,  $1<r<\infty$  and  $\alpha\leq r,$
\begin{eqnarray*}
\left(\frac{1}{C_{1}^{*}}[\mathfrak{b},T](a)\chi_{Q}, Q\right)\in\mathcal{M}\ell(p,\alpha,r,s-\lfloor\mu_{b}\rfloor-1) 
\end{eqnarray*}
where $C_{1}^{*}>0$ is a constant independent of  $a.$ 
\end{lem}

\proof
Let $\lfloor\mu_{b}\rfloor +\lfloor d(\frac{1}{p}-1)\rfloor +1\leq s\leq \lfloor \frac{\delta -d-3}{2}\rfloor$,  $(a,Q)\in \mathcal{A}(p,\alpha,\infty,d+2s+2)$ and $1<r<+\infty$  with $\alpha\leq r.$  Put $\widetilde{Q}:=2\sqrt{d}Q$ and denote by $x_{Q}$ and  $\ell_{Q}$ respectively the center and the size length of $Q.$  We have:
\begin{eqnarray*}
\left\|[\mathfrak{b},T](a)\chi_{\widetilde{Q}}\right\|_{r}\leq\left\|[\mathfrak{b},T](a)\right\|_{r}\leq C_{b,r}|Q|^{\frac{1}{r}-\frac{1}{\alpha}}
\end{eqnarray*}
since $[\mathfrak{b},T]$ is bounded on $L^{r}$.

Fix $x\notin \widetilde{Q}.$ We have 
\begin{eqnarray*}
|[\mathfrak{b},T](a)(x)|&=&|T((\mathfrak{b}(x)-\mathfrak{b})a)(x)|\\
&\leq& \int_{Q}\bigg[B_{s}C_{s}|K(x,y)-K(x,x_{Q})| + |R(y,x_{Q})| \bigg] |\mathfrak{b}(x)-\mathfrak{b}(y)||a(y)|dy
\end{eqnarray*}
 thanks to  Taylor's formula  and  Estimates (\ref{A-A}) and  (\ref{jojo2bis}).  But then 
$$|K(x,y)-K(x,x_{Q})|\leq C \frac{|y-x_{Q}|^{\delta}}{|x-x_{Q}|^{d+\delta}}\leq C\frac{\ell^{\delta}_Q}{|x-x_{Q}|^{d+\delta}}$$
and   
$$ |R(y,x_{Q})|\leq
 C_{d,s}\frac{|y-x_{Q}|^{d+2s+3}}{|x-x_{Q}|^{2d+2s+3}}\leq  C_{d,s}\frac{\ell^{d+2s+3}_Q}{|x-x_{Q}|^{2d+2s+3}}$$
 for all $y\in Q$. We also have that    
\begin{eqnarray*}    
\int_{Q}|\mathfrak{b}(x)-\mathfrak{b}(y)||a(y)|dy\leq C\bigg(\frac{\ell_{Q}^{-\mu_{b}}}{|x-x_{Q}|^{-\mu_{b}}} +\left\|\mathfrak{b}\right\|_{BMO}\bigg) |Q
|^{-\frac{1}{\alpha}+1}.
\end{eqnarray*}
 It comes that 
$|[\mathfrak{b},T](a)(x)|\leq\frac{1}{4} C_{1}^{*}\bigg(E_{1}+E_{2}+E_{3}+E_{4}\bigg)|Q|^{-\frac{1}{\alpha}}$
with 

$E_{1}= \frac{\ell_{Q}^{d+\delta-\mu_{b}}}{|x-x_{Q}|^{d+\delta-\mu_{b}}}, \   E_{2}=\frac{\ell_{Q}^{d+\delta}}{|x-x_{Q}|^{d+\delta}}, \ E_{3}=\frac{\ell_{Q}^{2d+2s+3-\mu_{b}}}{|x-x_{Q}|^{2d+2s+3-\mu_{b}}}\text{ and }E_{4}=\frac{\ell_{Q}^{2d+2s+3}}{|x-x_{Q}|^{2d+2s+3}}$, where $C_{1}^{*}:= C(r,b,s,T)$. 
Since $|x-x_{Q}|\approx \ell_{Q}+ |x-x_{Q}|$ it comes that 
 \begin{eqnarray*}
 E_{i}\lesssim \left(\dfrac{\ell_{Q}+ |x-x_{Q}|}{\ell_{Q}}\right)^{-2d-2s-3+ 2(\lfloor \mu_{b}\rfloor+1)}.
  \end{eqnarray*}
 for all $i\in\{1;2;3;4\}$ and all $x\notin \widetilde{Q}.$

Consequently, 
\begin{eqnarray*}
|[\mathfrak{b},T](a)(x)|\leq C_{1}^{*} |Q|^{-\frac{1}{\alpha}}\left(\frac{\ell_{Q}+ |x-x_{Q}|}{\ell_{Q}}\right)^{-2d-2s-3+ 2(\lfloor \mu_{b}\rfloor+1)}. 
\end{eqnarray*}
 Hence 
$$\left\|\frac{1}{C_{1}^{*}}[\mathfrak{b},T](a)\chi_{Q}\right\|_{r}\leq |Q|^{\frac{1}{r}-\frac{1}{\alpha}}$$ and 
$$\bigg|\frac{1}{C_{1}^{*}}[\mathfrak{b},T](a)(x)\bigg|\leq |Q|^{-\frac{1}{\alpha}}\left(1+\dfrac{ |x-x_{Q}|}{\ell_{Q}}\right)^{-2d-2(s-\lfloor \mu_{b}\rfloor-1)-3}$$
for all $x\notin \widetilde{Q}.$
We also have 
$$\int_{\mathbb{R}^{d}}x^{\beta}[\mathfrak{b},T](a)(x)dx=0,$$
for all $|\beta|\leq s-\lfloor \mu_{b}\rfloor-1;$ since $a\in L^{\nu}
(\mathbb{R}^{d})$ for $\nu\geq \max\{2,r\} >1$ and $\int_{\mathbb{R}^{d}}x^{\beta}(a)(x)dx=0,$ for $|\beta|\leq d+2s+2.$ 

Thus  $\frac{1}{C_{1}^{*}}[\mathfrak{b},T](a)$ is a $(q,\alpha,r,s-\lfloor \mu_{b}\rfloor-1)$-molecule centered on $Q.$

\epf

\proof[Proof of Theorem \ref{goupiste4}]
Let $f\in \mathcal{H}_{\mathrm{fin},\infty,s_{0}}^{(p,q,\alpha)}\cap \mathcal{C}(\mathbb{R}^{d})$ where $s_{0}=d+2s+2$. Consider a finite sequence  $\{(a_{n},Q_{n})\}_{n=0}^{j}$ of elements of $\mathcal{A}(p,\alpha,\infty,d+2s+2)$ and a finite sequence of scalars $\{\lambda_{n}\}_{n=0}^{j}$ such that $f=\sum_{n=0}^{j}\lambda_{n}a_{n}.$

Let $\max\{1,q\}<r<\infty.$ Accoding to  Lemma \ref{goupiste3}, $\frac{1}{C_{1}^{*}}[\mathfrak{b},T](a_{n})$  is a $(p,\alpha,r,s-\lfloor\mu_{b}\rfloor-1)$-molecule centered on $Q_{n},$ for $n\in\{0,1,...,j\}.$ We Fix $0<\eta<p.$ We have

\begin{eqnarray*}
[\mathfrak{b},T](f)\leq\sum_{n=0}^{j}\lambda_{n}[\mathfrak{b},T](a_{n})=\sum_{n=0}^{j}(c_{1}^{*}\lambda_{n})\left(\frac{1}{c_{1}^{*}}[\mathfrak{b},T](a_{n})\right) \in \mathcal{H}^{(p,q,\alpha)}
\end{eqnarray*}
and
\begin{eqnarray*}
\left\|[\mathfrak{b},T](f)\right\|_{\mathcal{H}^{(p,q,\alpha)}}\lesssim \left\|\sum_{n=0}^{j}\left(\frac{|c_{1}^{*}\lambda_{n}|}{\left\|\chi_{Q_{n}}\right\|_{\alpha}}\right)^{\eta}\chi_{Q_{n}} \right\|_{\frac{p}{\eta},\frac{q}{\eta},\frac{\alpha}{\eta}}^{\frac{1}{\eta}}\lesssim\left\|\sum_{n=0}^{j}\left(\frac{|\lambda_{n}|}{\left\|\chi_{Q_{n}}\right\|_{\alpha}}\right)^{\eta}\chi_{Q_{n}} \right\|_{\frac{p}{\eta},\frac{q}{\eta},\frac{\alpha}{\eta}}^{\frac{1}{\eta}}
\end{eqnarray*}
thanks to Theorem \ref{reconstmolecule2}. We deduce  that 
\begin{eqnarray*}
\left\|[\mathfrak{b},T](f)\right\|_{\mathcal{H}^{(p,q,\alpha)}}\lesssim\left\|f\right\|_{\mathcal{H}_{\mathrm{fin},\infty,s_{0}}^{(p,q,\alpha)}}\lesssim \left\|f\right\|_{\mathcal{H}^{(p,q,\alpha)}}.
\end{eqnarray*}
Therefore, $[\mathfrak{b},T]$ is a bounded operator from  $\mathcal{H}_{\mathrm{fin},\infty,s_{0}}^{(p,q,\alpha)}(\mathbb R^d)\cap \mathcal{C}(\mathbb{R}^{d})$ into $\mathcal{H}^{(p,q,\alpha)}$ and the density of $\mathcal{H}_{\mathrm{fin},\infty,s_{0}}^{(p,q,\alpha)}(\mathbb R^d)\cap \mathcal{C}(\mathbb{R}^{d})$ in $\mathcal{H}^{(p,q,\alpha)}(\mathbb R^d)$ yields the result.
\epf

\end{document}